\documentclass[11pt]{amsart}
\usepackage{mathrsfs}
\usepackage{amssymb}
\usepackage{graphicx}
\usepackage[T1]{fontenc}
\usepackage[all]{xy}
\usepackage{arydshln}
\usepackage{amscd}
\usepackage{amsmath, amssymb}
\usepackage{amsthm}
\usepackage{amsfonts}
\usepackage[colorlinks,linkcolor=blue,citecolor=blue, pdfstartview=FitH]
{hyperref}
\usepackage{backref}
\usepackage{amscd}
\usepackage{easybmat}
\usepackage{mathrsfs}
\usepackage{amsfonts}
\usepackage{color}
\usepackage{pifont}
\usepackage{upgreek}
\usepackage{bm}
\usepackage{hyperref}
\usepackage{shorttoc}
\usepackage{amsmath,amstext,amsthm,a4,amssymb,amscd}
\usepackage[mathscr]{eucal}
\usepackage{mathrsfs}
\usepackage{epsf}
\usepackage{tikz}
\usepackage{CJK}

\numberwithin{equation}{section}

\def\mc{\mathcal}

\def\n{\nabla}

\theoremstyle{plain}
\newtheorem{thm}{Theorem}[section]

\newtheorem{cor}[thm]{Corollary}
\theoremstyle{definition}
\newtheorem{rem}[thm]{Remark}
\newtheorem{con}[thm]{Conjecture}

\theoremstyle{definition}
\newtheorem{defn}[thm]{Definition}

\newcommand{\comment}[1]{}

\usepackage{amscd}

\usepackage{fancyhdr}
\makeatletter
\@namedef{subjclassname@2020}{\textup{2020} Mathematics Subject Classification}
\makeatother

\begin{document}
\begin{CJK}{UTF8}{gbsn}

\title{Positivity of simplicial volume for closed nonpositively curved four-manifolds with nonzero Euler characteristic}
\author{Inkang Kim}
\address{Inkang Kim: School of Mathematics, KIAS, Heogiro 85, Dongdaemun-gu Seoul, 02455, Republic of Korea}
\email{inkang@kias.re.kr}

\author{Xueyuan Wan}
\address{Xueyuan Wan: Mathematical Science Research Center, Chongqing University of
Technology, Chongqing 400054, China}
\email{xwan@cqut.edu.cn}

\date{\today}

\begin{abstract}

In this paper, by employing the Gauss-Bonnet theorem for Riemannian simplices due to Allendoerfer and Weil, we show that if a closed nonpositively curved $4$-manifold has nonzero Euler characteristic, then its simplicial volume is necessarily positive.
This result partially resolves conjectures posed by Connell-Ruan-Wang and Gromov concerning the relationship between the simplicial volume and the Euler characteristic for four-dimensional manifolds. As an application, we show that if a closed nonpositively curved $4$-manifold has negative Ricci curvature, then its simplicial volume is positive, thereby confirming in dimension four another conjecture of Gromov on the positivity of simplicial volume.

\end{abstract}

 

\maketitle

\section{Introduction}

Let $M$ be a closed, connected, oriented $n$-dimensional manifold. The simplicial volume of $M$, first introduced by Gromov \cite{Gromov_1982}, is defined as
\[
\|M\| = \inf \left\{\sum_i |a_i| : \left[\sum_i a_i \sigma_i\right] = [M] \in H_n(M, \mathbb{R})\right\},
\]
where the infimum is taken over all real singular cycles representing the fundamental homology class of $M$.

The positivity of simplicial volume has been extensively studied. For hyperbolic manifolds, the simplicial volume is proportional to their hyperbolic volume; see Gromov \cite[Page 11]{Gromov_1982} and Thurston \cite[Theorem 6.2]{Thurston}. For closed oriented manifolds of strictly negative curvature, the simplicial volume is also known to be positive; see Gromov \cite[Page 10]{Gromov_1982} and Inoue-Yano \cite[Theorem 1]{Inoue_Yano_1982}. For closed locally symmetric spaces of non-compact type, positivity is also known:
see Lafont and Schmidt \cite[Theorem 1.1]{Lafont_Schmidt_2006}, and Bucher-Karlsson
\cite[Theorem 1]{BK07} for the case of spaces covered by $\mathrm{SL}_3(\mathbb{R})/\mathrm{SO}(3)$. More generally, Gromov proposed the following conjecture (see \cite{MR670930} or also \cite[Page 11]{Gromov_1982}):

\begin{con}\label{con1}
If $M$ admits a Riemannian metric with nonpositive sectional curvature and negative definite Ricci curvature, then $\|M\| > 0$.
\end{con}

Regarding this conjecture, Connell and Wang \cite[Theorem 1.7]{Connell_Wang_2020} provided a purely differential geometric proof in dimension three, and also showed the positivity of simplicial volume under the assumption that the $\lfloor n/4\rfloor+1$-Ricci curvature is negative at one point \cite[Theorem 1.9]{Connell_Wang_2020}. 
Recently, Connell, Ruan, and Wang \cite{CRW2024} proved that the simplicial volume of a closed nonpositively curved manifold is positive if it admits an isolated, closed totally geodesic codimension-one submanifold.
Additionally, there is a close relationship between simplicial volume and the Euler characteristic. For instance, if a closed manifold $M$ supports an affine flat bundle of dimension $n=\mathrm{dim} M$, then it holds that $\|M\| \geq |\chi|$, where $\chi$ is the Euler number of the bundle (see \cite{MR95518, MR418119} and also \cite[Page 10]{Gromov_1982}). More generally, Gromov posed the following conjecture in \cite[Page 232]{MR1253544}:

\begin{con}\label{con2}
If $M$ is aspherical and $\|M\| = 0$, then $\chi(M) = 0$.
\end{con}

For recent progress and potential approaches regarding the above conjecture, we refer the reader to \cite{MR4455175, loeh2023}.
In particular, for closed nonpositively curved $4$-manifolds, Connell, Ruan, and Wang \cite[Corollary 5.1]{CRW23} showed that Conjecture \ref{con2} implies Conjecture \ref{con1} by observing that the Ricci curvature must degenerate at some point if the Euler characteristic vanishes. Furthermore, they proposed the following equivalence \cite[Conjecture 4]{CRW23}:

\begin{con}\label{con3}
Let $M$ be a closed nonpositively curved $4$-manifold. Then $\|M\| = 0$ if and only if $\chi(M) = 0$.
\end{con}

In the case where $M$ is a real analytic nonpositively curved $4$-manifold, Connell, Ruan, and Wang \cite[Corollary 1]{CRW23} proved that $\chi(M) = 0$ implies $\|M\| = 0$ by using a result of Schroeder.

In this paper, by applying the Gauss-Bonnet theorem for Riemannian simplices as established by Allendoerfer and Weil \cite{Carl-Weil}, we prove the following main theorem, which partially resolves Conjectures~\ref{con3} and~\ref{con2} for closed nonpositively curved $4$-manifolds.
\begin{thm}\label{main thm}
Let $M$ be a closed nonpositively curved $4$-manifold. Then
\[
  \|M\|\geq \frac{1}{11}|\chi(M)|.
\]
In particular, if $\chi(M)\neq 0$, then $\|M\|>0$.
\end{thm}

Theorem \ref{main thm} proves, in particular, the implication
\[
\|M\|=0 \Longrightarrow \chi(M)=0
\]
for every closed nonpositively curved $4$-manifold $M$, and in fact yields the
stronger quantitative estimate
\(\|M\| \ge \tfrac{1}{11} |\chi(M)|.\)
This should be compared with the result of Connell, Ruan, and Wang
\cite[Corollary 1]{CRW23}, who proved the converse implication
\[
\chi(M)=0 \Longrightarrow \|M\|=0
\]
under the additional assumption that $M$ is real analytic. Combining the two
results, we obtain the following consequence, which establishes Conjecture~\ref{con3}
in the real analytic case.

\begin{cor}
Let $M$ be a closed real analytic nonpositively curved $4$-manifold. Then $\|M\|=0$ if and only if $\chi(M)=0$.
In particular, Conjecture~\ref{con3} holds for closed real analytic nonpositively curved
$4$-manifolds.
\end{cor}

As another application of Theorem \ref{main thm}, by \cite[Corollary 5.1]{CRW23}, we resolve Conjecture \ref{con1} for dimension four:
\begin{cor}
If $M$ is a closed $4$-dimensional manifold admitting a Riemannian metric with nonpositive sectional curvature and negative definite Ricci curvature, then $\|M\| > 0$.
\end{cor}

On the other hand, as noted in \cite[Page 10]{Gromov_1982}, if $M_1$ and $M_2$ are closed $n$-dimensional manifolds with $n \geq 3$, then
\begin{equation*}
 \|M_1 \# M_2\| = \|M_1\| + \|M_2\|, \quad \|M_1 \times M_2\| \geq  \tfrac{1}{C}\|M_1\| \|M_2\|,
\end{equation*}
where $C > 0$ is a constant depending only on $\dim(M_1 \times M_2)$.

By Theorem \ref{main thm}, we obtain the following corollary:

\begin{cor}
Let $M_1$ be a closed nonpositively curved $4$-manifold with nonzero Euler characteristic. Then:
\begin{itemize}
  \item[(i)] $\|M_1 \# M_2\| > 0$ for any closed $4$-manifold $M_2$;
  \item[(ii)] $\|M_1 \times M_2\| > 0$ for any closed manifold $M_2$ with nonzero simplicial volume.
\end{itemize}
\end{cor}

Homology spheres and open book decomposition are important issues to 4-dimensional topologists. As an application of Theorem~\ref{main thm}, there exist $4$-dimensional integral homology spheres with positive simplicial volume; see \cite{MR2114711} and \cite[Proposition~18]{Kotschick2025}. Moreover, a nonpositively curved Riemannian manifold with nonzero Euler characteristic cannot admit an open book decomposition \cite[Theorem~B]{kastenholz2025}.

\

\text{{\bfseries{Acknowledgments.}}} 
The authors would like to thank the anonymous referees for their careful reading and valuable comments, which helped improve the paper.
The second author is grateful to Professors Haiping Fu, Shi Wang, and Francesco Milizia for many helpful discussions. The first author gratefully acknowledges the hospitality and support of the Max Planck Institute during his visit.
Research by Inkang Kim is partially supported by RS-2026-25468457 and KIAS Individual Grant (MG031408). Research by Xueyuan Wan is sponsored by the National Key R\&D Program of China (Grant No. 2024YFA1013200).

\

\section{Positivity of simplicial volume}

In this section, we first recall the definitions of simplicial volume and geodesic simplices. We then review the Gauss-Bonnet theorem for Riemannian simplices established by Allendoerfer and Weil. Finally, we present the proof of our main theorem.

\subsection{Definitions of simplicial volume and geodesic simplices}\label{sec-sim-geo}

We first recall the definitions of simplicial volume and geodesic simplices. For further details, see also~\cite{Lafont_Schmidt_2006, Gromov_1982}.

\subsubsection{Simplicial volume } The definition of simplicial volume is as follows.
\begin{defn}
Let $M$ be a topological space and $C^0(\Delta^k, M)$ be the set of singular $k$-simplices. For a singular real chain $\sigma = \sum_{i=1}^N a_i \sigma_i \in C_k(M,\mathbb{R})$, where each $a_i \in \mathbb{R}$ and $\sigma_i \in C^0(\Delta^k, M)$, the $l^1$-norm of $\sigma$ is defined by
\[
\|\sigma\| = \sum_{i=1}^N |a_i|.
\] 
The $l^1$-(pseudo)norm of a real singular homology class $\alpha \in H_k(M, \mathbb{R})$ is defined by
\[
\|\alpha\| = \inf \left\{\|\sigma\|: \partial\sigma=0 \text{ and }[\sigma]=\alpha\right\}.
\]
\end{defn}

\begin{defn}
Let $M^n$ be an oriented, closed, connected $n$-manifold. The \emph{simplicial volume} of $M^n$ is defined as
\[
\|M^n\|=\|i([M^n])\|,
\] 
where $i: H_n(M, \mathbb{Z}) \rightarrow H_n(M, \mathbb{R})$ is the change of coefficients homomorphism, and $[M^n]$ denotes the fundamental class arising from the orientation of $M^n$.
\end{defn}

\begin{rem}
This invariant is multiplicative under finite covers, thus its definition can be naturally extended to all closed, non-orientable manifolds.
\end{rem}

\subsubsection{Geodesic simplices}\label{geodesic-simplex}
Let $(M,g)$ be a closed Riemannian manifold with nonpositive sectional curvature. Consider the universal cover $p:\tilde{M}\to M$, which is a local isometry when $\tilde{M}$ is endowed with the pullback metric $p^*g$. We revisit the concept of a geodesic $k$-simplex; see also~\cite[Section~1]{Inoue_Yano_1982} and \cite[Section 2.2]{MR2499443}.

We use nonhomogeneous coordinates for the standard simplex throughout the paper:
\[
\Delta^k
=
\left\{(t^1,\ldots,t^k)\in\mathbb{R}^k:
 t^\alpha\geq 0,\ \sum_{\alpha=1}^k t^\alpha\leq 1\right\},
\]
with $\Delta^0$ consisting of one point. We identify $\Delta^{k-1}$ with the face
$\{(t^1,\ldots,t^k)\in\Delta^k:t^k=0\}$ and write $e_k$ for the $k$-th standard basis vector of $\mathbb{R}^k$.

Given an ordered $(k+1)$-tuple of points $\{p_0,\ldots,p_k\}$ in $\tilde M$, we define the
geodesic simplex $\sigma_{p_0\cdots p_k}:\Delta^k\to \tilde M$ inductively as follows.
For $k=0$, $\sigma_{p_0}$ is the constant map with image $p_0$. Assume that
$\sigma_{p_0\cdots p_{k-1}}:\Delta^{k-1}\to \tilde M$ has been defined. Then
$\sigma_{p_0\cdots p_k}$ is defined by joining each point of the image of
$\sigma_{p_0\cdots p_{k-1}}$ to $p_k$ by the unique geodesic segment, parametrized
proportionally to arc length. More precisely, identifying
\[
\Delta^k=\{(1-t)x+t e_k \mid x\in \Delta^{k-1},\ t\in[0,1]\},
\]
we set
\begin{equation*}
\sigma_{p_0\cdots p_k}\bigl((1-t)x+t e_k\bigr)
\end{equation*}
to be the point at time $t$ on the geodesic segment from
$\sigma_{p_0\cdots p_{k-1}}(x)$ to $p_k$. By \cite[Proposition 2.4]{MR2499443}, each geodesic $k$-simplex is smooth since the exponential map is a diffeomorphism.

Given a singular simplex $\sigma\colon \Delta^k \to M$, define $\mathrm{str}(\sigma):=p\circ \tilde{\sigma}$, where $\tilde{\sigma}$ is the geodesic $k$-simplex in $\tilde{M}$ determined by the vertices of a lift of $\sigma$. We call $\mathrm{str}(\sigma)$ the geodesic straightening of $\sigma$, or simply a geodesically straight simplex.
 Then $\mathrm{str}$ extends linearly to a map
\[
\mathrm{str}:C_*(M,\mathbb{R})\to C_*(M,\mathbb{R})
\]
satisfying:
\begin{itemize}
  \item $\mathrm{str}$ is a chain map which is chain homotopic to the identity;
  \item $\|\mathrm{str}(\sigma)\|\leq \|\sigma\|$ for all $\sigma\in C_*(M,\mathbb{R})$.
\end{itemize}
Hence, for any $\alpha\in H_*(M,\mathbb{R})$, we have 
\begin{equation*}
  \|\alpha\|=\inf\{\|\sigma\| : \sigma \in \mathrm{str}(C_*(M,\mathbb{R}))\text{ represents }\alpha\}.
\end{equation*}

In particular, the simplicial volume can be expressed as
\begin{equation*}
  \|M\|=\|i([M])\|=\inf\{\|\sigma\| : \sigma \in \mathrm{str}(C_*(M,\mathbb{R}))\text{ represents } i([M])\in H_n(M,\mathbb{R})\}.
\end{equation*}

\subsection{Gauss-Bonnet theorem for Riemannian simplices}\label{sec-2}

In this section, we will recall the Gauss-Bonnet theorem for Riemannian simplices \cite{Carl-Weil}; see also \cite[Section 2]{McMullen}.

Let $S^n \subset \mathbb{R}^{n+1}$ denote the unit sphere, and let $B^{n+1}$ denote the unit ball.
We denote by
\begin{equation}\label{area-sphere}
  \omega_n=2\cdot\tfrac{(4\pi)^{n/2}\Gamma(\tfrac{n}{2}+1)}{n!}
\end{equation}
the volume of $S^n$.
Let $A$ be a smooth Riemannian manifold of dimension $n$. Choose a local
coordinate chart $(x^1,\dots,x^n)$ near $x\in A$, and write $g_{ij}$ for the
metric tensor, $g=\det(g_{ij})$, and $R_{ijkl}$ for the components of the
Riemann curvature tensor in these local coordinates.
 Here and throughout the paper, $S_n$ denotes the symmetric group on
$\{1,\ldots,n\}$, namely the set of all permutations of $\{1,\ldots,n\}$.
For $i=(i_1,\ldots,i_n)\in S_n$, we view $i$ as the permutation sending
$t$ to $i_t$ for each $t\in\{1,\ldots,n\}$, and similarly for
$j=(j_1,\ldots,j_n)\in S_n$.
When $n$ is even, the
intrinsic Gauss-Bonnet integrand on $A$ is given by
\begin{equation}\label{eqn4}
  \Psi_n(x)=\tfrac{2}{\omega_n}\cdot\tfrac{1}{2^{n/2}n!}\sum_{i,j\in S_n}\tfrac{\epsilon(i)\epsilon(j)}{g}R_{i_1i_2j_1j_2}\cdots R_{i_{n-1}i_nj_{n-1}j_n}.
\end{equation}
 Here $\epsilon(i)$ and $\epsilon(j)$ denote the signs of the permutations $i$ and $j$. When $n$ is odd, we set $\Psi_n(x)=0$. The Gauss-Bonnet theorem states that for any closed manifold $A$ we have
\begin{equation}\label{eq2.4}
\chi(A)=\int_A \Psi_n(x)\,dv(x),
\end{equation}
where $dv(x)$ is the volume form on $A$ induced by the Riemannian metric.

Now let $A$ be an $r$-dimensional submanifold of a Riemannian manifold $B$ of
dimension $n$, where $r<n$. Let $R_{ijkl}$ denote the restriction of the Riemann curvature tensor on $B$ to $A$, and let $\Lambda_{ij}(\xi)$ denote the second fundamental form, a symmetric tensor on $A$ depending linearly on a normal vector $\xi$. In local coordinates, we have
\begin{equation*}
  \Lambda_{ij}(\xi)=\langle\n_{e_i}e_j,\xi\rangle.
\end{equation*}
The extrinsic Gauss-Bonnet integrand is the function on the unit normal bundle to $A$ defined by
\begin{equation}\label{eqn6}
  \Psi_r(x,\xi)=\sum_{0\leq 2f\leq r}\Psi_{r,f}(x,\xi),
\end{equation}
where 
\begin{align}\label{eqn10}
\begin{split}
  \Psi_{r,f}(x,\xi)=&\tfrac{2}{\omega_{2f}\omega_{n-2f-1}}\cdot\tfrac{1}{2^f(2f)!(r-2f)!}\cdot\sum_{i,j\in S_r}\tfrac{\epsilon(i)\epsilon(j)}{\gamma}\\
  &\times R_{i_1i_2j_1j_2}\cdots R_{i_{2f-1}i_{2f}j_{2f-1}j_{2f}}\Lambda_{i_{2f+1}j_{2f+1}}(\xi)\cdots\Lambda_{i_rj_r}(\xi).
 \end{split}
\end{align}
Here $\gamma$ is the determinant of the
induced metric on $A$. When it is necessary to record the ambient manifold, we write
$\Psi_r^{A\subset B}(x,\xi)$ for the extrinsic integrand in \eqref{eqn6}--\eqref{eqn10}
associated with the embedding $A^r\subset B^n$, and we write $\Psi_n^A(x)$ for the
intrinsic integrand in \eqref{eqn4}. When the submanifold and the ambient manifold are
clear, these superscripts are suppressed. This convention is applied facewise when $A$
is a union of faces.

We adopt the convention that empty products are equal to
$1$. In particular, the edge cases of \eqref{eqn10} are interpreted as follows:
\[
\Psi_{0,0}(x,\xi)=\frac{1}{\omega_{n-1}},
\]
and
\[
\Psi_{1,0}(x,\xi)=\frac{1}{\omega_{n-1}\gamma}\Lambda_{11}(\xi).
\]
These special cases will be used later in the low-dimensional computations.

The intrinsic and extrinsic integrands on $A$ are related by 
\begin{equation}\label{eqn9}
  \Psi_r(x)=\int_{S(x)}\Psi_r(x,\xi)d\xi,
\end{equation}
where $S(x)$ is the unit sphere (of dimension $n-r-1$) in the normal bundle to $A$ at $x$, and $\int_{S(x)}d\xi=\omega_{n-r-1}$. In particular, if $A$ is a closed submanifold of $B$ then we have 
\begin{equation*}
  \chi(A)=\int_A dv(x)\int_{S(x)}\Psi_r(x,\xi)d\xi.
\end{equation*}

Let  \( M \) be a simplex equipped with a smooth Riemannian metric. For any \( r < n \), let \( M[r] \) denote the union of all \( r \)-dimensional faces of \( \partial M \), and let \( M[n] = M \setminus \partial M \). Since $M$ is a smooth Riemannian simplex, we may extend its Riemannian metric to a small neighborhood of $M$. In this way, each face of $M$ can be regarded as a smooth submanifold of this neighborhood, and hence its normal bundle and its extrinsic Gauss--Bonnet integrand are well-defined. For each point \( x \in M[r] \), we define
\[
N(x) \subset S(x)
\]
as the convex set of unit vectors normal to \( M[r] \) that point inward toward \( M[n] \).

Since \( M \) is a Riemannian simplex, we may isometrically embed \( M \hookrightarrow \mathbb{R}^{N+1} \) for some sufficiently large \( N \). Let \( T \subset \mathbb{R}^{N+1} \) be the boundary of a small tubular neighborhood of \( M \), i.e., the set of points at fixed distance \( \delta > 0 \) from \( M \). Denote by 
\[ \mc{G}: T\to S^N \] 
the Gauss map. For a sufficiently small \( \delta \), we also have a well-defined nearest-point projection \( \pi: T \to M \).

Let \( T[r] = \pi^{-1}(M[r]) \subset T \) be the subset of \( T \) that lies closest to \( M[r] \). Then for any \( t \in T[r] \), we have a corresponding vector \( \xi = \mc{G}(t) \in S^N \) and a point \( x = \pi(t) \in M[r] \). 
Using this correspondence and Weyl's tube formula \cite{Weyl1939}, one finds that the contributions to the degree of the Gauss map are given by
\begin{equation}\label{eqn1}
\mathcal{G}(M[n]) = \tfrac{1}{\omega_N} \int_{T[n]} \mc{G}^*(d\xi) = \int_{M[n]} \Psi_n(x)\, dv(x),
\end{equation}
and for \( r < n \),
\begin{equation}\label{eqn2}
\mathcal{G}(M[r]) = \tfrac{1}{\omega_N} \int_{T[r]} \mc{G}^*(d\xi) = \int_{M[r]} dv(x) \int_{N(x)^*} \Psi_r(x, \xi)\, d\xi,
\end{equation}
where $\omega_N$ denotes the volume of the unit sphere $S^N$. For each
$x\in M[r]$, we view
\[
T_x(M[r]) \subset T_x\mathbb{R}^{N+1}\cong \mathbb{R}^{N+1}
\]
as the tangent space of the $r$-face $M[r]$ at $x$, regarded as a linear subspace
of the ambient Euclidean space. Thus $N(x)\subset S(x)\subset (T_x(M[r]))^\perp$
is a convex subset of the unit sphere in the orthogonal complement of $T_x(M[r])$.
We define
\begin{equation}\label{eqn25}
N(x)^*
:=
\left\{
\xi\in (T_x(M[r]))^\perp : |\xi|=1,\ \langle \xi,v\rangle\ge 0
\ \text{for all } v\in N(x)
\right\}.
\end{equation}
Equivalently, $N(x)^*$ is the spherical dual of the inward normal cone $N(x)$
inside the unit sphere $S(x)\subset (T_x(M[r]))^\perp$.

Remarkably, these quantities are independent of the choice of isometric embedding \( M \hookrightarrow \mathbb{R}^{N+1} \).
In particular, for $r = n-1$, the fiber $N(x)^*$ consists of only the unit inward normal vector. Therefore,
\begin{equation}\label{eqn13}
  \mathcal{G}(M[n-1]) = \int_{M[n-1]} \Psi_{n-1}(x, \xi) \, dv(x),
\end{equation}
where $\xi=\xi(x)$ denotes the unit inward normal vector.

 Since \( M \) is a simplex, the Gauss map has degree one; the Gauss-Bonnet theorem for a Riemannian simplex is given as follows:
\begin{equation}\label{eqn8}
 \mathcal{G}(M[n])+ \mathcal{G}(M[n-1])+\cdots+ \mathcal{G}(M[1])+ \mathcal{G}(M[0])=1.
\end{equation}

\begin{rem}\label{rem1}
More generally, consider a pair of smoothly embedded submanifolds
\(A \subset B \subset C\). For each \(x \in A\), let \(S_B(x)\) and
\(S_C(x)\) denote the fibers of the unit normal bundles to \(A\),
viewed as a submanifold of \(B\) and of \(C\), respectively. A convex
subset \(K \subset S_B(x) \subset S_C(x)\) then has two duals,
namely \(K_B^*=\bigl(K, S_B(x)\bigr)^*\) and
\(K_C^*=\bigl(K, S_C(x)\bigr)^*\). Likewise, in the notation introduced after \eqref{eqn10}, the two extrinsic
Gauss--Bonnet integrands are \(\Psi_r^{A\subset B}\) and
\(\Psi_r^{A\subset C}\), respectively. In this notation we have the basic
relation
\begin{equation}\label{eqn17}
  \int_{K_B^*} \Psi_r^{A\subset B}(x, \xi)\,d\xi
  = \int_{K_C^*} \Psi_r^{A\subset C}(x, \xi)\,d\xi.
\end{equation}
Equation \eqref{eqn9} can be regarded as a special case of \eqref{eqn17}, with $A=B$ and $K=\emptyset$.
\end{rem}

\subsubsection{Gauss-Bonnet theorem for $n=2$}\label{sec-app1}

Suppose $M$ is a two-dimensional Riemannian simplex. By \eqref{eqn4}, we have
\begin{equation}\label{eq:GB-2-simplex-1}
  \Psi_2(x)=\tfrac{1}{2\pi}\tfrac{R_{1212}}{g}=\tfrac{1}{2\pi}K,
\end{equation}
where $K:=R_{1212}/g$ denotes the Gaussian curvature. Moreover, by \eqref{eqn6} and \eqref{eqn10}, we have
\[
\Psi_1(x,\xi)=\Psi_{1,0}(x,\xi)=\tfrac{1}{2\pi\gamma}\Lambda_{11}(\xi), \quad \Psi_0(x,\xi)=\Psi_{0,0}(x,\xi)=\tfrac{1}{2\pi}.
\]
Thus, by \eqref{eqn8}, the Gauss-Bonnet theorem for $M$ takes the form
\[
  \tfrac{1}{2\pi}\int_{M[2]}K\,dv(x) 
  + \tfrac{1}{2\pi}\int_{M[1]}dv(x)\int_{N(x)^*}\tfrac{1}{\gamma}\Lambda_{11}(\xi)\,d\xi 
  + \tfrac{1}{2\pi}\sum_{i=1}^3\mathrm{vol}(N(x_i)^*)=1.
\]
In particular, if $M$ is a geodesic simplex, then we have $\Lambda_{11}(\xi)=0$. Therefore, the Gauss-Bonnet formula simplifies to
\[
 \int_{M[2]}K\,dv(x)+\sum_{i=1}^3\mathrm{vol}(N(x_i)^*)=2\pi.
\]
Observe that each $\mathrm{vol}(N(x_i)^*)=\alpha_i\in(0,\pi]$ is precisely the exterior angle at vertex $x_i$. Thus, we conclude
\begin{equation}\label{eq:GB-2-simplex}
   \int_{M[2]}K\,dv(x)=2\pi-(\alpha_1+\alpha_2+\alpha_3)=-\pi+\sum_{i=1}^3(\pi-\alpha_i)\geq -\pi.
\end{equation}

\subsubsection{Gauss-Bonnet integrands for $n=4$}\label{sec-app2}
Now suppose $M$ is a four-dimensional Riemannian simplex. In the later proof of Theorem~1.4, only the formulas for $\Psi_0$, $\Psi_1$,
and $\Psi_2$ will be used essentially. We also include the expressions for
$\Psi_3$ and $\Psi_4$ here for completeness. By equations \eqref{eqn4}, \eqref{eqn6}, and \eqref{eqn10}, we have
\begin{equation*}
\Psi_0(x,\xi)=\Psi_{0,0}(x,\xi)=\tfrac{1}{2\pi^2},\quad  
\Psi_{1}(x,\xi)=\Psi_{1,0}(x,\xi)=\tfrac{1}{2\pi^2\gamma}\Lambda_{11}(\xi),
\end{equation*}
\begin{equation}\label{eqn24}
\Psi_{2}(x,\xi)=\tfrac{R_{1212}+2\,\operatorname{det}\Lambda(\xi)}{4\pi^2\gamma},\quad
\Psi_{3}(x,\xi)=\tfrac{\operatorname{det}\Lambda(\xi)}{2\pi^2\gamma}+\tfrac{1}{16\pi^2\gamma}\varepsilon^{abc}\varepsilon^{pqr}R_{abpq}\Lambda_{cr}(\xi),
\end{equation}
and
\[
\Psi_4(x)=\tfrac{1}{32\pi^2}\left(R_{ijkl}R_{ijkl}-4R_{ij}R_{ij}+R^2\right)=\tfrac{1}{32\pi^2}\left(|R|_g^2-4|Ric|_g^2+R_g^2\right).
\]

\subsubsection{Gauss-Bonnet integrands for $n=8$}\label{sec-app8}

If $n=8$, by \eqref{eqn6} and \eqref{eqn10}, we obtain
\[
\Psi_0(x,\xi)
 = \Psi_{0,0}(x,\xi)
 = \frac{3}{\pi^4},
\qquad
\Psi_1(x,\xi)
 = \Psi_{1,0}(x,\xi)
 = \frac{3}{\pi^4\,\gamma}\,\Lambda_{11}(\xi),
\]
and
\[
\Psi_2(x,\xi)
 = \Psi_{2,0}(x,\xi)+\Psi_{2,1}(x,\xi)
 = \frac{R_{1212}+6\,\operatorname{det}\Lambda(\xi)}{2\pi^4\,\gamma}.
\]
Here $\Lambda(\xi)=(\Lambda_{ij}(\xi))_{1\le i,j\le 2}$ when $r=2$, so that
$\det\Lambda(\xi)=\Lambda_{11}(\xi)\Lambda_{22}(\xi)-\Lambda_{12}(\xi)^2$.

\subsection{Proof of main theorem}
In this section, we will prove Theorem \ref{main thm}.
\begin{proof}[Proof of Theorem \ref{main thm}]
Let $(M,g)$ be a closed nonpositively curved $4$-manifold. By Section \ref{geodesic-simplex}, the simplicial volume of $M$ can be computed using geodesically straight simplices. More precisely,
\[
\|M\|
=
\inf\left\{
\begin{array}{l}
\|\sigma\| : \sigma \text{ is a real singular }4\text{-cycle representing }[M],\text{and each} \\ 
\text{simplex of }\sigma \text{ is the projection of a geodesic }4\text{-simplex in }\tilde M
\end{array}
\right\},
\]
see also \cite[Section~1]{Inoue_Yano_1982} and \cite[Proposition~2.12]{MR2499443}.

If $\omega$ is a smooth $p$-form on $M$ and $\sigma$ is any smooth $p$-simplex in $M$, then the integral of $\omega$ over $\sigma$ is defined by
\begin{equation}\label{eqn11}
\int_\sigma \omega := \int_{\Delta^p} \sigma^* \omega.
\end{equation}
In particular, if $\sigma \colon \Delta^p \to M$ is an embedding and $\dim M=p$, then
\begin{equation}\label{embedding-inter}
\int_\sigma \omega
=
\int_{\Delta^p} \sigma^*\omega
=
\epsilon_\sigma \int_{\sigma(\Delta^p)} \omega,
\end{equation}
where $\epsilon_\sigma$ denotes the degree of the map $\sigma\colon \Delta^p\to \sigma(\Delta^p)$. More precisely, $\epsilon_\sigma=1$ if $\sigma$ is orientation-preserving, and $\epsilon_\sigma=-1$ if $\sigma$ is orientation-reversing.
If $c=\sum_{i=1}^k c_i\sigma_i$ is a smooth $p$-chain, we define
\begin{equation}\label{eqn-chain-integral}
\int_c\omega
:=
\sum_{i=1}^k c_i\int_{\sigma_i}\omega
=
\int_{\Delta^p}\Bigl(\sum_{i=1}^k c_i\sigma_i^*\omega\Bigr)
=:
\int_{\Delta^p} c^*\omega,
\end{equation}
where $c^*\omega:=\sum_{i=1}^k c_i\sigma_i^*\omega$; see \cite[p.~481]{Lee_2013}.

Let $\tilde{\sigma}=\sum_{i=1}^N a_i\tilde{\sigma}_i$ be a chain in the universal cover $\tilde M$, where each $\tilde{\sigma}_i\colon \Delta^4\to \tilde M$ is a geodesic simplex. Let $\sigma:=p\circ \tilde{\sigma}$ be the corresponding chain in $M$, where $p\colon \tilde M\to M$ is the covering map, and assume that $\sigma$ represents the fundamental class $[M]$. By \cite[Proposition~2.4]{MR2499443}, each geodesic simplex $\tilde{\sigma}_i$ is smooth. Hence $\sigma=p\circ \tilde{\sigma}$ is a smooth singular $4$-chain in $M$. Writing $\sigma_i:=p\circ \tilde{\sigma}_i$, we have $\sigma=\sum_{i=1}^N a_i\sigma_i$, and since $\sigma$ represents $[M]$, it follows that
\begin{equation}\label{eqn16}
\partial \sigma=\sum_{i=1}^N a_i\,\partial \sigma_i=0.
\end{equation}

By the Cartan--Hadamard theorem, $\tilde M$ is diffeomorphic to $\mathbb{R}^4$ via the exponential map. We therefore identify $\tilde M$ with $\mathbb{R}^4$ and regard $(\tilde M,\tilde g:=p^*g)$ as an isometrically and totally geodesically embedded submanifold of $(\mathbb{R}^8,g_8)$, where $g_8=\tilde g\oplus g_{\mathrm{eucl}}$ and $g_{\mathrm{eucl}}$ denotes the standard Euclidean metric on $\mathbb{R}^4$.

A geodesic simplex $\tilde{\sigma}_i$ need not be an immersion. Nevertheless, after an arbitrarily small perturbation, we may construct, for each $i$, a one-parameter family $\{\tilde{\sigma}_i^\varepsilon\}_{\varepsilon>0}$ with $\tilde{\sigma}_i^\varepsilon\colon \Delta^4\to \mathbb{R}^8$ such that every $\tilde{\sigma}_i^\varepsilon$ is a smooth embedding and $\tilde{\sigma}_i^\varepsilon\to \tilde{\sigma}_i$ as $\varepsilon\to 0$. Indeed, under the identification $\tilde M\cong \mathbb{R}^4$, we may define
\begin{equation}\label{eqn-pertur}
\tilde{\sigma}_i^\varepsilon(t^1,\ldots,t^4)
=
\bigl(\tilde{\sigma}_i(t^1,\ldots,t^4),
\varepsilon t^1,\ldots,\varepsilon t^4\bigr)
\in \mathbb{R}^8\equiv \tilde M\times\mathbb{R}^4,
\end{equation}
where $t=(t^1,\ldots,t^4)\in\Delta^4\subset\mathbb{R}^4$.
It is immediate that $\tilde{\sigma}_i^\varepsilon$ is an embedding for every $\varepsilon>0$, and that $\tilde{\sigma}_i^\varepsilon\to \tilde{\sigma}_i$ as $\varepsilon\to 0$.

For later use, set
\[
A_i^\varepsilon:=\tilde\sigma_i^\varepsilon(\Delta^4)\subset\mathbb{R}^8.
\]
We write
\[
\tilde\Psi_{4,i}^\varepsilon:=\Psi_4^{A_i^\varepsilon}
\]
for the intrinsic Gauss--Bonnet integrand of $A_i^\varepsilon$ with its induced metric,
and write $\tilde\Psi_4:=\Psi_4^{\tilde M}$ for the intrinsic integrand of
$(\tilde M,\tilde g)$. Let $dv_g$ and $dv_{\tilde g}$ denote the volume forms of
$(M,g)$ and $(\tilde M,\tilde g)$, respectively. For each $\varepsilon>0$, let
$dv_{\tilde{\sigma}_i^\varepsilon}$ denote the volume form of the embedded Riemannian
manifold $A_i^\varepsilon$ satisfying
$(\tilde{\sigma}_i^\varepsilon)^*(dv_{\tilde{\sigma}_i^\varepsilon})\to
\tilde{\sigma}_i^*(dv_{\tilde g})$ as $\varepsilon\to 0$.

By~\eqref{eq2.4}, we obtain
\begin{equation}\label{eqn7}
\begin{aligned}
\chi(M)
&=\int_M \Psi_4(x)\,dv_g(x) \\
&=\sum_{i=1}^N a_i \int_{\tilde{\sigma}_i}\tilde{\Psi}_4(\tilde x)\,dv_{\tilde g}(\tilde x) \\
&=\lim_{\varepsilon\to 0}\sum_{i=1}^N a_i \int_{\tilde{\sigma}_i^\varepsilon}\tilde\Psi_{4,i}^\varepsilon(\tilde x)\,dv_{\tilde{\sigma}_i^\varepsilon}(\tilde x) \\
&=\lim_{\varepsilon\to 0}\sum_{i=1}^N a_i\,\epsilon_{\tilde{\sigma}_i^\varepsilon}
\int_{\tilde{\sigma}_i^\varepsilon(\Delta^4)}\tilde\Psi_{4,i}^\varepsilon(\tilde x)\,dv_{\tilde{\sigma}_i^\varepsilon}(\tilde x).
\end{aligned}
\end{equation}
Here the second equality follows from the fact that $\sigma=\sum_{i=1}^N a_i\sigma_i$ represents the fundamental class $[M]$. Indeed, $\Psi_4(x)\,dv_g(x)$ is a closed $4$-form on $M$, so its integral over $M$ agrees with its pairing with any smooth singular $4$-cycle representing $[M]$. Therefore,
\[
\int_M \Psi_4(x)\,dv_g(x)
=
\sum_{i=1}^N a_i \int_{\sigma_i} \Psi_4(x)\,dv_g(x).
\]
Since $\sigma_i=p\circ \tilde{\sigma}_i$ and $p^*(\Psi_4\,dv_g)=\tilde{\Psi}_4\,dv_{\tilde g}$, we further obtain
\[
\sum_{i=1}^N a_i \int_{\sigma_i} \Psi_4(x)\,dv_g(x)
=
\sum_{i=1}^N a_i \int_{\tilde{\sigma}_i} \tilde{\Psi}_4(\tilde{x})\,dv_{\tilde g}(\tilde{x}),
\]
which is exactly the second equality in \eqref{eqn7}. Since each $\tilde{\sigma}_i^\varepsilon$ is an embedding, \eqref{embedding-inter} applies, and the integrand $\tilde\Psi_{4,i}^\varepsilon(\tilde x)$ in the last line is given by \eqref{eqn4} in terms of the curvature tensor of the induced metric on $\tilde{\sigma}_i^\varepsilon(\Delta^4)$.

We next analyze the boundary term. The idea is to show that the corresponding extrinsic Gauss--Bonnet integrals over the $3$-faces of $\tilde{\sigma}_i^\varepsilon(\Delta^4)$ vanish. Combining this with \eqref{eqn7} and \eqref{eqn8}, we can then express $\chi(M)$ as a sum of extrinsic Gauss--Bonnet integrals over the $0$-, $1$-, and $2$-faces. Finally, we estimate these lower-dimensional contributions by using the special properties of geodesic simplices.

The Riemannian manifold $\bigl(\tilde{\sigma}_i^\varepsilon(\Delta^4),g_8|_{\tilde{\sigma}_i^\varepsilon(\Delta^4)}\bigr)$ induces both a metric and an orientation on its boundary. Let $dv_{\partial\tilde{\sigma}_i^\varepsilon}$ denote the induced volume form on the boundary $\partial(\tilde{\sigma}_i^\varepsilon(\Delta^4))$. If $\xi_0$ denotes the inward unit normal vector field along the boundary, then
\begin{equation}\label{eqn-or}
dv_{\partial\tilde{\sigma}_i^\varepsilon}
=
i^*_{\partial(\tilde{\sigma}_i^\varepsilon(\Delta^4))}
\bigl((-\xi_0)\lrcorner\, dv_{\tilde{\sigma}_i^\varepsilon}\bigr).
\end{equation}
See, e.g., \cite[Corollary~15.34]{Lee_2013}. Here $i_{\partial(\tilde{\sigma}_i^\varepsilon(\Delta^4))}:\partial(\tilde{\sigma}_i^\varepsilon(\Delta^4))\hookrightarrow \tilde{\sigma}_i^\varepsilon(\Delta^4)$ denotes the natural inclusion.
For each open $3$-face $F\subset\partial A_i^\varepsilon$, define
\[
\tilde\Psi_{3,i}^\varepsilon\big|_F
:=
\Psi_3^{F\subset A_i^\varepsilon}.
\]
Thus $\tilde\Psi_{3,i}^\varepsilon(\tilde x,\xi_0)$ denotes the extrinsic
Gauss--Bonnet integrand of the $3$-face containing $\tilde x$, viewed as a
hypersurface of $A_i^\varepsilon$ and evaluated at its inward unit normal $\xi_0$.
All boundary integrands and integrals below are understood facewise; intersections of
distinct $3$-faces have $3$-dimensional measure zero.
The corresponding boundary integral is
\[
\int_{\partial\tilde{\sigma}_i^\varepsilon}
\tilde\Psi_{3,i}^\varepsilon(\tilde x,\xi_0)\,
dv_{\partial\tilde{\sigma}_i^\varepsilon}(\tilde x).
\]

By definition,
\[
\partial\tilde{\sigma}_i^\varepsilon=\sum_{k=0}^4(-1)^k\tilde{\sigma}_i^\varepsilon\circ d_k^4,
\]
where $d_k^q\colon \Delta^{q-1}\to \Delta^q$ is the standard face map. With the
nonhomogeneous-coordinate convention for $\Delta^q$ fixed in
Section~\ref{geodesic-simplex}, we define
\[
d_k^q(t^1,\dots,t^{q-1})=(t^1,\dots,t^{k-1},0,t^k,\dots,t^{q-1})
\]
for $1\le k\le q$, and
\[
d_0^q(t^1,\dots,t^{q-1})=(1-\sum_{\alpha=1}^{q-1}t^\alpha,\ t^1,\dots,t^{q-1}).
\]
Let $F_k:=d_k^q(\Delta^{q-1})\subset \partial\Delta^q$. Then $F_k$ carries the boundary orientation induced from $\Delta^q$.
We now determine the degree of $d_k^q$ with respect to this orientation. For $k\ge 1$, we have $F_k=\{t^k=0\}\subset \partial\Delta^q$, whose outward unit normal is $-\frac{\partial}{\partial t^k}$. Hence the induced orientation form on $F_k$ is
\[
\omega_k
=
-\tfrac{\partial}{\partial t^k}\,\lrcorner\,\Omega_q
=
(-1)^k\,dt^1\wedge\cdots\wedge \widehat{dt^k}\wedge\cdots\wedge dt^q,
\]
where $\Omega_q=dt^1\wedge\cdots\wedge dt^q$ denotes the standard volume form on $\Delta^q$.
Therefore $(d_k^q)^*\omega_k=(-1)^k\Omega_{q-1}$, so $\epsilon_{d_k^q}=(-1)^k$ for $k\ge 1$. For $k=0$, the face $F_0$ is given by
\(
F_0=\left\{(t^1,\dots,t^q)\in \Delta^q:\sum_{\alpha=1}^q t^\alpha=1\right\}.
\)
An outward normal vector is $\sum_{\alpha=1}^q \frac{\partial}{\partial t^\alpha}$. Hence the boundary orientation form induced from $\Delta^q$ is
\[
\omega_0
=
(\sum_{\alpha=1}^q \tfrac{\partial}{\partial t^\alpha})\!\lrcorner\,
\Omega_q
=
\sum_{\alpha=1}^q (-1)^{\alpha-1}
dt^1\wedge\cdots\wedge \widehat{dt^\alpha}\wedge\cdots\wedge dt^q.
\]
A direct calculation shows that
\(
(d_0^q)^*\omega_0
=
q\Omega_{q-1}.
\)
In particular, $(d_0^q)^*\omega_0$ is a positive multiple of the standard orientation form on $\Delta^{q-1}$. Hence $d_0^q$ is orientation-preserving, and therefore $\epsilon_{d_0^q}=1$. Altogether,
\[
\epsilon_{d_k^q}=(-1)^k,\qquad 0\le k\le q.
\]

Moreover, the restriction $\tilde{\sigma}_i^\varepsilon|_{F_k}\colon F_k\to \partial(\tilde{\sigma}_i^\varepsilon(\Delta^4))$ has the same orientation sign as $\tilde{\sigma}_i^\varepsilon$. Consequently,
\[
\epsilon_{\tilde{\sigma}_i^\varepsilon\circ d_k^4}
=
\epsilon_{\tilde{\sigma}_i^\varepsilon|_{F_k}}\epsilon_{d_k^4}
=
\epsilon_{\tilde{\sigma}_i^\varepsilon}(-1)^k.
\]
It follows that
\begin{equation}\label{eqn21}
\begin{aligned}
\int_{\partial\tilde{\sigma}_i^\varepsilon}
\tilde\Psi_{3,i}^\varepsilon(\tilde x,\xi_0)\,dv_{\partial\tilde{\sigma}_i^\varepsilon}(\tilde x)
&=\sum_{k=0}^4(-1)^k
\int_{\tilde{\sigma}_i^\varepsilon\circ d_k^4}
\tilde\Psi_{3,i}^\varepsilon(\tilde x,\xi_0)\,dv_{\partial\tilde{\sigma}_i^\varepsilon}(\tilde x) \\
&=\sum_{k=0}^4(-1)^k\epsilon_{\tilde{\sigma}_i^\varepsilon\circ d_k^4}
\int_{\tilde{\sigma}_i^\varepsilon\circ d_k^4(\Delta^3)}
\tilde\Psi_{3,i}^\varepsilon(\tilde x,\xi_0)\,dv_{\partial\tilde{\sigma}_i^\varepsilon}(\tilde x) \\
&=\epsilon_{\tilde{\sigma}_i^\varepsilon}
\int_{\partial(\tilde{\sigma}_i^\varepsilon(\Delta^4))}
\tilde\Psi_{3,i}^\varepsilon(\tilde x,\xi_0)\,dv_{\partial\tilde{\sigma}_i^\varepsilon}(\tilde x).
\end{aligned}
\end{equation}

We claim that
\begin{equation}\label{eqn22}
\lim_{\varepsilon\to 0}
\sum_{i=1}^N a_i
\int_{\partial\tilde{\sigma}_i^\varepsilon}
\tilde\Psi_{3,i}^\varepsilon(\tilde x,\xi_0)\,dv_{\partial\tilde{\sigma}_i^\varepsilon}(\tilde x)
=0.
\end{equation}
By \eqref{embedding-inter}, one has
\[
\lim_{\varepsilon\to 0}
\sum_{i=1}^N a_i
\int_{\partial\tilde{\sigma}_i^\varepsilon}
\tilde\Psi_{3,i}^\varepsilon(\tilde x,\xi_0)\,dv_{\partial\tilde{\sigma}_i^\varepsilon}(\tilde x)
=
\lim_{\varepsilon\to 0}
\int_{\Delta^3}
\sum_{i=1}^N a_i
(\partial\tilde{\sigma}_i^\varepsilon)^*
\bigl(\tilde\Psi_{3,i}^\varepsilon(\tilde x,\xi_0)\,dv_{\partial\tilde{\sigma}_i^\varepsilon}(\tilde x)\bigr).
\]
By Remark \ref{rem1} and \eqref{eqn17}, if $F$ is the open $3$-face of
$A_i^\varepsilon$ containing $\tilde x$, then
\[
\tilde\Psi_{3,i}^\varepsilon(\tilde{x},\xi_0)
=
\int_{\{\xi_0\}_{\mathbb{R}^8}^{*}}
\Psi_3^{F\subset\mathbb{R}^8}(\tilde{x},\xi)\,d\xi.
\]
Here $\{\xi_0\}_{\mathbb{R}^8}^{*}$ is the dual set of $\{\xi_0\}$ with respect to the
inclusion $F\subset\mathbb{R}^8$; see Remark~\ref{rem1} and \eqref{eqn25}.
Each term in the expansion of
$\Psi_3^{F\subset\mathbb{R}^8}(\tilde{x},\xi)\,\gamma$ is a product of
curvature components and components of the second fundamental form
$\Lambda_{ab}(\xi)=\langle\nabla_{e_a}e_b,\xi\rangle$. More precisely,
\begin{align*}
\begin{split}
\Psi_3^{F\subset\mathbb{R}^8}(\tilde{x},\xi)\,\gamma
&=
\sum_{0\le 2f\le 3}
\frac{2}{\omega_{2f}\omega_{8-2f-1}}
\cdot
\frac{1}{2^f(2f)!(3-2f)!}
\sum_{\alpha,\beta\in S_3}\epsilon(\alpha)\epsilon(\beta) \\
&\quad\times
R_{\alpha_1\alpha_2\beta_1\beta_2}\cdots
R_{\alpha_{2f-1}\alpha_{2f}\beta_{2f-1}\beta_{2f}}
\Lambda_{\alpha_{2f+1}\beta_{2f+1}}(\xi)\cdots
\Lambda_{\alpha_3\beta_3}(\xi) \\
&=
\frac{\det(\Lambda(\xi))}{\omega_7}
+
\frac{1}{8\pi\omega_5}
\sum_{\alpha,\beta\in S_3}\epsilon(\alpha)\epsilon(\beta)\,
R_{\alpha_1\alpha_2\beta_1\beta_2}
\Lambda_{\alpha_3\beta_3}(\xi).
\end{split}
\end{align*}
Exactly as in the proof of \eqref{uniform-bound},
$\Psi_3^{F\subset\mathbb{R}^8}(\tilde{x},\xi)\,\gamma$ is bounded on the
sphere tangent bundle of $\tilde M\times I_0$ for any compact subset
$I_0\subset\mathbb{R}^4$. On the other hand, one can choose such a compact subset $I_0\subset \mathbb{R}^4$ so that, for every $\varepsilon>0$, the corresponding dual set $\{\xi_0\}_{\mathbb{R}^8}^{*}$ is contained in the sphere tangent bundle of $\tilde{M}\times I_0$.  Therefore, after integration over the
dual set, the form
$\tilde\Psi_{3,i}^\varepsilon(\tilde{x},\xi_0)\,
dv_{\partial\tilde\sigma_i^\varepsilon}(\tilde{x})$ is uniformly bounded.

Moreover, since the boundary maps $\tilde{\sigma}_i^\varepsilon\circ d_k^4\colon \Delta^3\to \partial(\tilde{\sigma}_i^\varepsilon(\Delta^4))$ form a smooth family on the compact domain $\Delta^3$ and converge to $\tilde{\sigma}_i\circ d_k^4$ as $\varepsilon\to 0$, the pullback forms
\(
(\partial\tilde{\sigma}_i^\varepsilon)^*
\bigl(
\tilde\Psi_{3,i}^\varepsilon(\tilde{x},\xi_0)\,
dv_{\partial\tilde{\sigma}_i^\varepsilon}(\tilde{x})
\bigr)
\)
are also uniformly bounded on $\Delta^3$. Equivalently, if $\Omega_3:=dt^1\wedge dt^2\wedge dt^3$ denotes the standard volume form on $\Delta^3$, then the coefficient functions 
$$(\partial\tilde{\sigma}_i^\varepsilon)^*
\bigl(
\tilde\Psi_{3,i}^\varepsilon(\tilde{x},\xi_0)\,
dv_{\partial\tilde{\sigma}_i^\varepsilon}(\tilde{x})
\bigr)/\Omega_3$$
are uniformly bounded, independently of $i$ and $\varepsilon$.

By the dominated convergence theorem, in order to prove \eqref{eqn22}, it suffices to show that for every point $t\in \Delta^3$, one has
\begin{equation}\label{eqn26}
\lim_{\varepsilon\to 0}
\sum_{i=1}^N a_i
(\partial\tilde{\sigma}_i^\varepsilon)^*
\bigl(\tilde\Psi_{3,i}^\varepsilon(\tilde{x},\xi_0)\,dv_{\partial\tilde{\sigma}_i^\varepsilon}(\tilde{x})\bigr)(t)
=0.
\end{equation}

Fix $t\in\Delta^3$. Since each $\tilde\sigma_i^\varepsilon$ is an embedding,
after shrinking to a sufficiently small neighborhood $U_t$ of $t$, every map
$\sigma_i^\varepsilon\circ d_k^4$ is an embedding on $U_t$, where
$\sigma_i^\varepsilon:=\tilde p\circ\tilde\sigma_i^\varepsilon$ and
$\tilde p=(p,\mathrm{Id})\colon\tilde M\times\mathbb{R}^4\to
M\times\mathbb{R}^4$. Put $B_i^\varepsilon:=\sigma_i^\varepsilon(\Delta^4)$.
On each local $3$-face $F\subset\partial B_i^\varepsilon$, write
\[
\Psi_{3,i}^\varepsilon\big|_F:=\Psi_3^{F\subset B_i^\varepsilon},
\]
and let $dv_{\partial B_i^\varepsilon}$ be the corresponding boundary volume form.
Because $\tilde p$ is a local isometry, on $U_t$ we have
\[
\sum_{i=1}^N a_i
(\partial\tilde\sigma_i^\varepsilon)^*
\bigl(\tilde\Psi_{3,i}^\varepsilon(\tilde x,\xi_0)\,
dv_{\partial\tilde\sigma_i^\varepsilon}(\tilde x)\bigr)
=
\sum_{i=1}^N a_i
(\partial\sigma_i^\varepsilon)^*
\bigl(\Psi_{3,i}^\varepsilon(x,\xi_0)\,
dv_{\partial B_i^\varepsilon}(x)\bigr).
\]

Let $\{\tau_1,\ldots,\tau_{N_1}\}$ be the collection of distinct singular
$3$-simplices among $\{\sigma_i\circ d_k^4\}_{i,k}$. For
$i\in\{1,\ldots,N\}$, $k\in\{0,\ldots,4\}$, and
$j\in\{1,\ldots,N_1\}$, define
\[
\delta_{ik,j}:=
\begin{cases}
1,&\text{if }\sigma_i\circ d_k^4=\tau_j,\\
0,&\text{otherwise.}
\end{cases}
\]
Because the $\tau_j$ are distinct, $\sum_{j=1}^{N_1}\delta_{ik,j}=1$ for every
$i$ and $k$. Moreover, the coefficient of each $\tau_j$ in
$\sum_{i=1}^N a_i\partial\sigma_i$ is zero by \eqref{eqn16}; hence
\begin{equation}\label{eqn-face-coeff}
\sum_{i=1}^N\sum_{k=0}^4 a_i(-1)^k\delta_{ik,j}=0
\qquad(1\le j\le N_1).
\end{equation}

For $\varepsilon>0$, define
\[
\tau_{jk}^\varepsilon\colon\Delta^3\longrightarrow M\times\mathbb{R}^4,
\qquad
\tau_{jk}^\varepsilon(u):=
\bigl(\tau_j(u),\varepsilon d_k^4(u)\bigr).
\]
The perturbation formula \eqref{eqn-pertur} gives
\[
\sigma_i^\varepsilon\circ d_k^4=\tau_{jk}^\varepsilon
\quad\Longleftrightarrow\quad
\sigma_i\circ d_k^4=\tau_j.
\]
Consequently,
\begin{align}\label{eqn-face-expansion}
\begin{split}
&\sum_{i=1}^N a_i
(\partial\sigma_i^\varepsilon)^*
\bigl(\Psi_{3,i}^\varepsilon(x,\xi_0)\,dv_{\partial B_i^\varepsilon}(x)\bigr)\\
&\qquad=
\sum_{i=1}^N\sum_{k=0}^4\sum_{j=1}^{N_1}
a_i(-1)^k\delta_{ik,j}
(\tau_{jk}^\varepsilon)^*
\bigl(\Psi_{3,i}^\varepsilon(x,\xi_0)\,dv_{\partial B_i^\varepsilon}(x)\bigr).
\end{split}
\end{align}

For each $j$, choose $i(j)$ and $k(j)$ such that
$\delta_{i(j)k(j),j}=1$, and set
\[
\Theta_j^\varepsilon
:=
(\tau_{jk(j)}^\varepsilon)^*
\bigl(\Psi_{3,i(j)}^\varepsilon(x,\xi_0)\,
dv_{\partial B_{i(j)}^\varepsilon}(x)\bigr).
\]
By \eqref{eqn-face-coeff},
\[
\sum_{i=1}^N\sum_{k=0}^4\sum_{j=1}^{N_1}
a_i(-1)^k\delta_{ik,j}\,\Theta_j^\varepsilon=0.
\]
Subtracting this identity from \eqref{eqn-face-expansion} gives
\[
\begin{aligned}
&\sum_{i=1}^N a_i
(\partial\sigma_i^\varepsilon)^*
\bigl(\Psi_{3,i}^\varepsilon(x,\xi_0)\,dv_{\partial B_i^\varepsilon}(x)\bigr)\\
&\quad=
\sum_{i=1}^N\sum_{k=0}^4\sum_{j=1}^{N_1}
a_i(-1)^k\delta_{ik,j}
\left[
(\tau_{jk}^\varepsilon)^*
\bigl(\Psi_{3,i}^\varepsilon(x,\xi_0)\,dv_{\partial B_i^\varepsilon}(x)\bigr)
-\Theta_j^\varepsilon
\right].
\end{aligned}
\]
Thus it is enough to prove the following pairwise assertion: if
$\delta_{i_1k_1,j}=\delta_{i_2k_2,j}=1$, then
\begin{equation}\label{eqn-pairwise}
\begin{split}
\lim_{\varepsilon\to0}\Bigl[&
(\tau_{jk_1}^\varepsilon)^*
\bigl(\Psi_{3,i_1}^\varepsilon(x,\xi_0)\,
dv_{\partial B_{i_1}^\varepsilon}(x)\bigr)(t)\\
&-
(\tau_{jk_2}^\varepsilon)^*
\bigl(\Psi_{3,i_2}^\varepsilon(x,\xi_0)\,
dv_{\partial B_{i_2}^\varepsilon}(x)\bigr)(t)
\Bigr]=0.
\end{split}
\end{equation}

We prove \eqref{eqn-pairwise}. Suppose first that
$\operatorname{rank}d\tau_j(t)<3$. The boundary $3$-forms above are uniformly
bounded, and
$d\tau_{jk}^\varepsilon(t)\to(d\tau_j(t),0)$ as $\varepsilon\to0$. Hence
\[
\left|
(\tau_{jk}^\varepsilon)^*
\bigl(\Psi_{3,i}^\varepsilon(x,\xi_0)\,
dv_{\partial B_i^\varepsilon}(x)\bigr)(t)
\right|
\le C\left|\bigwedge^3d\tau_{jk}^\varepsilon(t)\right|
\longrightarrow0
\]
whenever $\delta_{ik,j}=1$. Here \(\bigwedge^3 d\tau_{jk}^{\varepsilon}(t)\) denotes the linear map induced by \(d\tau_{jk}^{\varepsilon}(t)\) on third exterior powers.
Thus both terms in \eqref{eqn-pairwise} tend to zero.

Assume now that $\operatorname{rank}d\tau_j(t)=3$. After shrinking $U_t$ again,
$\tau_j|_{U_t}$ is an embedding. Every pair $(i,k)$ with $\delta_{ik,j}=1$
has the same unperturbed boundary germ $\tau_j(U_t)\subset M$. Along this
hypersurface a unit normal is determined up to sign. If $\nu$ is either choice,
set
\[
\eta_j
:=
\Psi_3^{\tau_j(U_t)\subset M}(x,\nu)\,
i_{\tau_j(U_t)}^*\bigl(({-\nu})\lrcorner\,dv_g\bigr).
\]
The explicit formula \eqref{eqn24} shows that
$\Psi_3^{\tau_j(U_t)\subset M}(x,\nu)$ is odd in $\nu$, while the displayed
boundary volume form is also odd in $\nu$. Therefore $\eta_j$ is independent
of the choice of sign. The smooth convergence of the perturbed faces and their
induced geometric data now gives, for every $(i,k)$ with $\delta_{ik,j}=1$,
\[
\lim_{\varepsilon\to0}
(\tau_{jk}^\varepsilon)^*
\bigl(\Psi_{3,i}^\varepsilon(x,\xi_0)\,
dv_{\partial B_i^\varepsilon}(x)\bigr)(t)
=
\tau_j^*\eta_j(t).
\]
Thus the two terms in \eqref{eqn-pairwise} have the same limit. This proves
\eqref{eqn-pairwise}, and hence \eqref{eqn26}. By dominated convergence,
\eqref{eqn22} follows.

By \eqref{eqn21}, it follows that
\begin{align*}
&\quad \lim_{\varepsilon\to 0}
\sum_{i=1}^N a_i\,\epsilon_{\tilde{\sigma}_i^\varepsilon}
\int_{\tilde{\sigma}_i^\varepsilon(\Delta^4)[3]}
\tilde\Psi_{3,i}^\varepsilon(\tilde x,\xi_0)\,dv_{\partial\tilde{\sigma}_i^\varepsilon}(\tilde x)\\
&=\lim_{\varepsilon\to 0}
\sum_{i=1}^N a_i\,\epsilon_{\tilde{\sigma}_i^\varepsilon}
\int_{\partial(\tilde{\sigma}_i^\varepsilon(\Delta^4))}
\tilde\Psi_{3,i}^\varepsilon(\tilde x,\xi_0)\,dv_{\partial\tilde{\sigma}_i^\varepsilon}(\tilde x)=0.
\end{align*}

For $k=0,1,2$ and for each open $k$-face $F\subset A_i^\varepsilon[k]$, define facewise
\[
\tilde\Psi_{k,i}^{\varepsilon,\mathbb{R}^8}\big|_F
:=
\Psi_k^{F\subset\mathbb{R}^8}.
\]
Let $N_{i,\varepsilon}(\tilde x)^*$ denote the corresponding dual normal set
in the unit normal sphere of $F\subset\mathbb{R}^8$. By Remark~\ref{rem1}, the
fiber integrals of these explicitly defined integrands agree with the
$0$-, $1$-, and $2$-face terms in the Gauss--Bonnet formula, independently of
the Euclidean isometric embedding used to derive that formula.

Combining this with \eqref{eqn7}, we obtain
\[
\begin{aligned}
\chi(M)
&=
\lim_{\varepsilon\to 0}
\sum_{i=1}^N a_i\,\epsilon_{\tilde{\sigma}_i^\varepsilon}
\left[
\int_{\tilde{\sigma}_i^\varepsilon(\Delta^4)}
\tilde\Psi_{4,i}^\varepsilon(\tilde x)\,dv_{\tilde{\sigma}_i^\varepsilon}(\tilde x)
+
\int_{\tilde{\sigma}_i^\varepsilon(\Delta^4)[3]}
\tilde\Psi_{3,i}^\varepsilon(\tilde x,\xi_0)\,dv_{\partial\tilde{\sigma}_i^\varepsilon}(\tilde x)
\right] \\
&=
-\lim_{\varepsilon\to 0}
\sum_{i=1}^N a_i\,\epsilon_{\tilde{\sigma}_i^\varepsilon}
\Biggl[
\int_{\tilde{\sigma}_i^\varepsilon(\Delta^4)[2]}
dv_{\tilde{\sigma}_i^\varepsilon[2]}(\tilde x)
\int_{N_{i,\varepsilon}(\tilde x)^*}\tilde\Psi_{2,i}^{\varepsilon,\mathbb{R}^8}(\tilde x,\xi)\,d\xi \\
&\qquad\qquad\qquad\qquad
+\int_{\tilde{\sigma}_i^\varepsilon(\Delta^4)[1]}
dv_{\tilde{\sigma}_i^\varepsilon[1]}(\tilde x)
\int_{N_{i,\varepsilon}(\tilde x)^*}\tilde\Psi_{1,i}^{\varepsilon,\mathbb{R}^8}(\tilde x,\xi)\,d\xi \\
&\qquad\qquad\qquad\qquad
+\int_{\tilde{\sigma}_i^\varepsilon(\Delta^4)[0]}
dv_{\tilde{\sigma}_i^\varepsilon[0]}(\tilde x)
\int_{N_{i,\varepsilon}(\tilde x)^*}\tilde\Psi_{0,i}^{\varepsilon,\mathbb{R}^8}(\tilde x,\xi)\,d\xi-1
\Biggr],
\end{aligned}
\]
where the second equality follows from the Gauss--Bonnet theorem for Riemannian simplices; see \eqref{eqn8}. Here $dv_{\tilde{\sigma}_i^\varepsilon[k]}$ denotes the induced volume form on the $k$-skeleton $\tilde{\sigma}_i^\varepsilon(\Delta^4)[k]$.

Since each $\tilde{\sigma}_i$ has five vertices and $\tilde\Psi_{0,i}^{\varepsilon,\mathbb{R}^8}(\tilde x,\xi)=\frac{3}{\pi^4}>0$ by Section \ref{sec-app8}, we have
\[
0\le
\int_{\tilde{\sigma}_i^\varepsilon(\Delta^4)[0]}
dv_{\tilde{\sigma}_i^\varepsilon[0]}(\tilde x)
\int_{N_{i,\varepsilon}(\tilde x)^*}\tilde\Psi_{0,i}^{\varepsilon,\mathbb{R}^8}(\tilde x,\xi)\,d\xi
=
\frac{1}{\omega_{N_0}}\int_{T_\varepsilon[0]}\mathcal{G}^*(d\xi)
\le 5,
\]
where the middle equality follows from \eqref{eqn2}. Here $N_0$ is chosen so that the Riemannian simplex $\tilde{\sigma}_i^\varepsilon(\Delta^4)$ admits an isometric embedding into some Euclidean space $\mathbb{R}^{N_0+1}$.

By the definition above and Remark~\ref{rem1}, the $k=1,2$ fiber integrals are
computed using the embeddings
\[
\tilde\sigma_i^\varepsilon(\Delta^4)[k]\subset\mathbb{R}^8.
\]
Moreover, the forms
\[
\tilde\Psi_{k,i}^{\varepsilon,\mathbb{R}^8}(\tilde x,\xi)\,
dv_{\tilde\sigma_i^\varepsilon[k]}(\tilde x)
\]
are continuous and uniformly bounded. Indeed, by Section~\ref{sec-app8},
\begin{equation}\label{uniform-bound}
\begin{aligned}
&\quad \bigl|\tilde\Psi_{2,i}^{\varepsilon,\mathbb{R}^8}(\tilde x,\xi)\,dv_{\tilde{\sigma}_i^\varepsilon[2]}(\tilde x)\bigr|=
\frac{1}{2\pi^4}\bigl|(R_{1212}+6\det\Lambda(\xi))\,e_1^*\wedge e_2^*\bigr| \\
&\le
\frac{1}{2\pi^4}
\max_{x\in M\times I_0}
\max_{\substack{e_1\wedge e_2\in \Lambda^2T_x(M\times I_0)\\ \|e_1\|=\|e_2\|=1}}
\Bigl(
|R(e_1,e_2,e_1,e_2)|
+
6\max_{\substack{\xi\perp(e_1\wedge e_2)\\ \|\xi\|=1}}
|\det\Lambda(\xi)|
\Bigr),
\end{aligned}
\end{equation}
where $I_0\subset\mathbb{R}^4$ is a compact subset such that all points of $\tilde{p}\circ\tilde{\sigma}_i^\varepsilon(\Delta^4)$ lie in $M\times I_0$. An analogous bound holds for $\bigl|\tilde\Psi_{1,i}^{\varepsilon,\mathbb{R}^8}(\tilde x,\xi)\,dv_{\tilde{\sigma}_i^\varepsilon[1]}(\tilde x)\bigr|$.

Furthermore, by Section \ref{sec-app8},
\[
\bigl|\tilde\Psi_{1,i}^{\varepsilon,\mathbb{R}^8}(\tilde x,\xi)\,dv_{\tilde{\sigma}_i^\varepsilon[1]}(\tilde x)\bigr|
=
\left|\frac{3}{\pi^4}\Lambda_{11}(\xi)\right|
\longrightarrow 0
\qquad\text{as }\varepsilon\to 0,
\]
since each $1$-face $\tilde{\sigma}_i[1]$ is geodesic. By the dominated convergence theorem,
\[
\lim_{\varepsilon\to 0}
\int_{\tilde{\sigma}_i^\varepsilon(\Delta^4)[1]}
dv_{\tilde{\sigma}_i^\varepsilon[1]}(\tilde x)
\int_{N_{i,\varepsilon}(\tilde x)^*}\tilde\Psi_{1,i}^{\varepsilon,\mathbb{R}^8}(\tilde x,\xi)\,d\xi
=0.
\]

For $\tilde\Psi_{2,i}^{\varepsilon,\mathbb{R}^8}$, Section \ref{sec-app8} gives $\tilde\Psi_{2,i}^{\varepsilon,\mathbb{R}^8}(\tilde x,\xi)=(R_{1212}+6\det\Lambda(\xi))/(2\pi^4\,\gamma)$. Choose a local orthonormal frame $\{e_1,e_2\}$ on the geodesic $2$-simplex $\tilde{\sigma}_i^\varepsilon[2]$ so that $e_1$ is tangent to a geodesic direction. Then $\Lambda_{11}(\xi)\to 0$ as $\varepsilon\to 0$, and hence 
$$\det\Lambda(\xi)=\Lambda_{11}(\xi)\Lambda_{22}(\xi)-\Lambda_{12}(\xi)^2\to -\Lambda_{12}(\xi)^2.$$
 Since the sectional curvature is nonpositive, we obtain
\[
\tilde\Psi_{2,i}^{\varepsilon,\mathbb{R}^8}(\tilde x,\xi)\,\gamma
\longrightarrow
\frac{R_{1212}-6\Lambda_{12}(\xi)^2}{2\pi^4}
\le
\frac{R_{1212}}{2\pi^4}
\le 0
\qquad\text{as }\varepsilon\to 0.
\]
Therefore, by the dominated convergence theorem,
\[
\begin{aligned}
0
&\le
\lim_{\varepsilon\to 0}
\int_{\tilde{\sigma}_i^\varepsilon(\Delta^4)[2]}
dv_{\tilde{\sigma}_i^\varepsilon[2]}(\tilde x)
\int_{N_{i,\varepsilon}(\tilde x)^*}\bigl(-\tilde\Psi_{2,i}^{\varepsilon,\mathbb{R}^8}(\tilde x,\xi)\bigr)\,d\xi \\
&\le
\lim_{\varepsilon\to 0}
\int_{\tilde{\sigma}_i^\varepsilon(\Delta^4)[2]}
dv_{\tilde{\sigma}_i^\varepsilon[2]}(\tilde x)
\int_{S(\tilde x)_\varepsilon}\bigl(-\tilde\Psi_{2,i}^{\varepsilon,\mathbb{R}^8}(\tilde x,\xi)\bigr)\,d\xi \\
&=
\lim_{\varepsilon\to 0}
\int_{\tilde{\sigma}_i^\varepsilon(\Delta^4)[2]}
\bigl(-\Psi_2^{F}(\tilde x)\bigr)\,dv_{\tilde{\sigma}_i^\varepsilon[2]}(\tilde x) \\
&=
\lim_{\varepsilon\to 0}
\int_{\tilde{\sigma}_i^\varepsilon(\Delta^4)[2]}
\Bigl(-\frac{1}{2\pi}K_{\tilde{\sigma}_i^\varepsilon[2]}\Bigr)\,dv_{\tilde{\sigma}_i^\varepsilon[2]}(\tilde x),
\end{aligned}
\]
where $F$ denotes the open $2$-face containing $\tilde x$; the first equality
follows from \eqref{eqn9}, and the second follows from
\eqref{eq:GB-2-simplex-1}. Here
$K_{\tilde{\sigma}_i^\varepsilon[2]}$ denotes the sectional curvature of the induced
metric on the $2$-face $\tilde{\sigma}_i^\varepsilon(\Delta^4)[2]$. Since each $4$-simplex $\tilde{\sigma}_i$ has $\binom{5}{3}=10$ two-dimensional faces, we obtain
\[
\lim_{\varepsilon\to 0}
\int_{\tilde{\sigma}_i^\varepsilon(\Delta^4)[2]}
\Bigl(-\frac{1}{2\pi}K_{\tilde{\sigma}_i^\varepsilon[2]}\Bigr)\,dv_{\tilde{\sigma}_i^\varepsilon[2]}(\tilde x)
\le
\frac{1}{2}\binom{5}{3}
=5,
\]
where the inequality follows from \eqref{eq:GB-2-simplex}, applied to each geodesic $2$-simplex.

Collecting the above estimates, we conclude that
\[
|\chi(M)|
\le
\left(\sum_{i=1}^N |a_i|\right)(1+5+5)
=
11\sum_{i=1}^N |a_i|.
\]
Taking the infimum over all real singular $4$-cycles representing $[M]$ whose simplices are projections of geodesic $4$-simplices in $\tilde M$, we obtain
\[
\|M\|
\ge
\frac{1}{11}|\chi(M)|.
\]
This completes the proof.
\end{proof}

\bibliographystyle{alpha}
\bibliography{simplicial}

\end{CJK}
\end{document}